\documentclass{birkjour}

\usepackage{amssymb,amsmath}
\usepackage{amsthm}
\usepackage{mathptmx} 
 \newtheorem{theorem}{Theorem}[section]
 
 \newtheorem{lemma}[theorem]{Lemma}
 
 \theoremstyle{definition}
 
 \theoremstyle{remark}
 
 \numberwithin{equation}{section}
\DeclareMathOperator{\md}{\mathsf{D}}

\newcommand{\pr}[1]{\mathsf P\left \{#1\right\}}
\newcommand{\E}[1]{\mathsf{E}\left[\,#1\,\right]}
\newcommand{\Ex}[1]{\mathsf{E}\Big[\,#1\,\Big]}
\renewcommand{\d}[1]{\ensuremath{\operatorname{d}\!{#1}}}

\newcommand{\R}{\mathbb R}
\newcommand{\ra}{\rangle}
\newcommand{\la}{\langle}
 
\newcommand{\norm}[1]{\left\lVert#1\right\rVert}
\newcommand{\abs}[1]{\left\lvert#1\right\rvert}

\newcommand*{\set}[1]{\left\{#1\right\}}
\newcommand{\ind}[1]{\mathbf{1}_{#1}}
\newcommand{\cI}{\mathcal I}

\newcommand{\eps}{\varepsilon}
\begin{document}
\title{Existence of density for solutions of mixed stochastic
equations}

%----------Author 1

%----------Author 2
\author{Taras Shalaiko}
\address{Department of Probability Theory, Statistics and Actuarial Mathematics,\br
Taras Shevchenko National University of Kyiv,\br
Volodymyrska 64, Kyiv 01601, Ukraine\br 
and\br
Department of Mathematical Economics II\br
Mathematical Institute of the Manhheim University\br
A5, 6\br
D-68131 Mannheim}
\email{tshalaik@mail.uni-mannheim.de}
%\thanks{The second author is supported by the DAAD ``Sandwich'' grant. }

\author{Georgiy Shevchenko}

\address{ Department of Probability Theory, Statistics and Actuarial Mathematics,\br
Taras Shevchenko National University of Kyiv,\br
Volodymyrska 64, Kyiv 01601, Ukraine}

\email{zhora@univ.kiev.ua}



\date{\today}

\begin{abstract}
We consider a mixed stochastic differential equation $\d{X_t}=a(t,X_t)\d{t}+b(t,X_t)\d{W_t}+c(t,X_t)\d{B^H_t}$ driven by independent multidimensional Wiener process and fractional Brownian motion. Under H\"ormander type conditions we show that the distribution of $X_t$ possesses a density with respect to the Lebesgue measure. 
\end{abstract}

%%% ----------------------------------------------------------------------
\maketitle
\section{Introduction}
In this paper we study a so-called mixed stochastic differential equation (SDE) in $\R^d$
\begin{equation}
X_t=X_0+\int_0^t a(s,X_s)\d s+\int_0^t b(s,X_s)\d{W_s}+\int_0^t c(s,X_s)\d{B^H_s}
\label{eqIntr}
\end{equation} 
driven by a multidimensional standard Wiener process and a multidimensional fractional Brownian motion (fBm) with Hurst parameter $H\in(1/2,1)$ (see next section for precise definitions). Recently such equations gained a lot of attention thanks to their modeling features. There is already a large literature devoted to them;  the few papers we cite here give an extensive overview of existing results. The unique solvability result in the form suitable for our needs is obtained in the paper \cite{Delay}; although the result is formulated there for equations with delay, it is a fortiori valid for usual equations. The paper \cite{Integr} contains useful estimates of the solution and results on its integrability. Finally, we mention the paper \cite{ShSh}, where the Malliavin differentiability of the solution is obtained.

The main aim of this article is to provide conditions under which the solution to \eqref{eqIntr} has a density with respect to the Lebesgue measure. For It\^o SDEs, such issues were addressed by many authors, see \cite{Nualart} and references therein. Existence and regularity of density for SDEs driven by fBm we proved in \cite{BauHai,NourdinSimon,NualSau} under H\"ormander type conditions. The recent paper \cite{Tindel} contains a generalization of these results to equations driven by Gaussian rough paths, in particular, it allows to deduce the existence of a smooth density of the solution to \eqref{eqIntr} with Stratonovich integral with respect to the Wiener process. However, the machinery used in that article is quite sophisticated, and here we use a more direct approach.

The paper is organised as follows. In Section 2 we introduce our notation, describe the main object and briefly discuss Malliavin calculus of variations for fractional Brownian motion. In Section 3, we prove that the distribution of the solution $X_t, t>0$ possesses density w.r.t. Lebesgue measure under a simplified version of the H\"ormander condition. Section 4 contains the result on existence and smoothness of the density under a strong version of the H\"ormander condition. The Appendix contains some technical lemmas and the Norris lemma for a mixed SDE.
\section{Preliminaries}
\subsection{Definitions and notation}

Throughout the paper, $\abs{\cdot}$ will denote the absolute value of a number, the Euclidean norm of a vector, and the operator norm of a matrix. $\la \cdot, \cdot\ra $ stays for the usual scalar product in the Euclidean space. We will use the symbol $C$ to denote a generic constant, whose value is not important and may change from one line to another. We will write a subscript if a constant is relevant or if its value depends on some parameters.

For a matrix $A=(a_{i,j})$ of arbitrary size, we denote by $a_i$ its $i$-th row and by $a_{\cdot,j}$ its $j$-th column.

The classes of continuous and $\theta$-H\"older continuous functions on $[a,b]$ will be denoted respectively by $C[a,b]$ and $C^\theta[a,b]$. For a function $f\colon [a,b]\to\R$ denote by $\norm{f}_{\infty,[a,b]}$ its supremum norm and by 
$$\norm{f}_{\theta,[a,b]}=\sup_{a\leq s<t\leq b}\frac{|f(t)-f(s)|}{|t-s|^\theta}$$
its $\theta$-H\"older seminorm. If there is no ambiguity, we will use the notation $\norm{f}_\infty$ and $\norm{f}_\theta$.

Finally, for a function $h\in C(\R^d)$ denote by $\partial_x h = (\frac{\partial}{\partial x_1}h,\dots, \frac{\partial}{\partial x_d}h)$ its gradient and by $\partial^2_{xx} h = (\frac{\partial^2}{\partial x_i\partial x_j}h)_{i,j=1,\dots,d}$ its second derivative matrix.

\subsection{Main equation and assumptions}
For a fixed time horizon $T>0$, let $\{\Omega, \mathcal F, \mathbf{F}=(\mathcal F_t)_{t\in [0,T]}, \mathsf P\}$ be a standard stochastic basis.  Equation \eqref{eqIntr} is driven by two independent sources of randomness:  an $m$-di\-men\-sional $\mathbf{F}$-Wiener process
$\{W_t=(W^1_t,\ldots,W^m_t), t\in [0,T]\}$ and an $l$-dimensional fBm $\{B^H_t=(B^{H,1}_t,\ldots,B^{H,l}_t), t\ge [0,T]\}$ with Hurst index $H\in(1/2,1)$, i.e. a centered Gaussian process having the covariance
$$\E{B^{H,i}_tB^{H,j}_s}=\frac{\delta_{i,j}}{2}(t^{2H}+s^{2H}-|t-s|^{2H}).$$
It is well known that the fBm $B^H$ has a modification with $\gamma$-H\"older continuous path for any $\gamma<H$, in the following we will assume that the process itself is H\"older continuous.

Equation \eqref{eqIntr} is understood as a system of  SDEs on $[0,T]$
\begin{equation*}
%\label{eq1}
\begin{aligned}
X^i_t=X^i_0  +\int_0^ ta_i(s,X_s)\d s+\sum_{j=1}^m\int_0^tb_{i,j}(s,X_s)\d W^j_s +\sum_{k=1}^l\int_0^tc_{i,k}(s,X_s)\d B^{H,k}_s, 
\end{aligned}
\end{equation*}
$i=1,\ldots,d$, with a non-random initial condition $X_0\in\R^d$. In this equation, the integral w.r.t. $W$ is understood in a usual It\^o sense, the one w.r.t. $B^H$ is understood in a pathwise sense, as Young integral. More information on its definition and properties can be found in \cite{friz-victoir}.

The coefficients $a_i,b_{i,j},c_{i,k}\colon [0,T]\times\R^d\to\R^d$, $i=1,\ldots,d,j=1,\ldots, m,k=1,\ldots, l$ are assumed to satisfy  the following conditions.
\begin{itemize}
\item[A1] for all $t\in[0,T]$ \ $a(t,\cdot),b(t,\cdot)\in C^1(\R^d)$, $c(t,\cdot)\in C^2(\R^d)$;
\item[A2] for all 
$t\in[0,T]$, $x\in \R^d$ 
$$\abs{a(t,x)} + \abs{b(t,x)} + \abs{c(t,x)}\le C(1+\abs{x});$$
\item[A3] for all 
$t\in[0,T]$, $x\in \R^d$ \ $\abs{\partial_x c(t,x)}\le C $;
\item[A4] there exists $\beta>0$ such that for all $t,s\in[0,T]$, $x\in\R^d$
$$
|c(t,x)-c(s,x)|\leq C|t-s|^\beta(1+|x|),\quad |\partial_x c(t,x)-\partial_x c(s,x)|\leq C|t-s|^\beta.
$$
\end{itemize}
The continuous differentiability implies that $a,b,\partial_x c$ are locally Lipschitz continuous. Therefore, by \cite[Theorem 4.1]{Delay}, equation \eqref{eqIntr} has a unique solution which is H\"older continuous of any order $\theta\in(0,1/2)$.

\subsection{Ad hoc Malliavin calculus}
Here we summarize some facts from the Malliavin calculus of variations, see \cite{Nualart} for a deeper exposition. Denote by $S[0,T]$ the set the of step functions of the form $f(t) = \sum_{k=1}^{n} c_k \ind{[a_k,b_k)}(t)$ defined on $[0,T]$.  
Let  $L^2_H[0,T]$ denote the separable Hilbert space obtained by completing $S[0,T]$ w.r.t.\ the scalar product
\begin{gather*}
\la f,g\ra_{L^2_H[0,T]}=\int_0^T\int_0^T f(t)g(s)\phi(t,s)\d t\d s,
\end{gather*}
where $\phi(t,s)=H(2H-1)|t-s|^{2H-2}$.

Consider the product space $$\mathfrak H = \left(L^2_H [0,T]\right)^{ l}\times \left(L^2[0,T]\right)^{m}.$$ 
It is also a separable Hilbert space with a scalar product 
$$\la f,g\ra_{\mathfrak H} = \sum_{i=1}^l \la f_i,g_i\ra_{L^2_H[0,T]} + \sum_{i=l+1}^{l+m} \la f_i,g_i\ra_{L^2[0,T]}.$$
%The process $\{(B^{H,1}_1,\ldots,B^{H,l}_t,W^{1}_t,\ldots,W^{m}_t)\}_{t\geq 0}$ is an isonormal Gaussian process with respect to $\mathfrak{H}$.
The map 
$$
\cI \colon (\ind{[0,t_1)},\dots,\ind{[0,t_l)},\ind{[0,s_1)},\dots,\ind{[0,s_m)})\mapsto (B^{H,1}_{t_1},\dots, B^{H,l}_{t_l},W^1_{s_1},\dots, W^m_{s_m})
$$
can be extended by linearity to $S[0,T]^{l+m}$. It appears that for $f,g\in S[0,T]^{l+m}$ 
$$\E{\la \cI(f),\cI(g)\ra} = \la f,g\ra_{\mathfrak H},$$
so $\cI$ can be extended to an isometry between $\mathfrak H$ and a subspace of $L^2(\Omega;\R^{m+l})$. 

For $\xi = F(\cI(f_1),\dots ,\cI(f_n))$, where $f_1,\dots,f_n\in \mathfrak H$ and $f_i=(f_{i,1},\ldots,f_{i,m+l})$, $i=1,\ldots,n$,  $F\colon \R^{n(m+l)} \to \R$ is a continuously differentiable finitely supported function, define the 
%changed the definition
Malliavin derivative $\mathsf D \xi$ as an element of $\mathfrak H$, whose  $j$-th coordinate equal to $$ \sum_{i=1}^{n} \partial_{(i-1)(l+m)+j} F (\cI(f_1),\dots, \cI(f_n)) f_{i,j}, j=1,\ldots,l+m.$$
Denote for $p\ge 1$ by $\mathbb D^{1,p}$ the closure of the space of smooth cylindrical random variables with respect to the norm 
$$\norm{\xi}^p_{\mathbb D^{1,p}}=\E{|\xi|^{p}+\norm{\md{\xi}}_{\mathfrak H}}^{1/p}.$$
$\md$ is closable in this space and its closure will be denoted likewise. Finally, the Malliavin derivative is a (possibly, generalized) function from $[0,T]$ to $\R^{l+m}$, so we can introduce the notation
$$
\md \xi = \set{\md_t\xi = \big(\md^{H,1}_t \xi,\dots, \md^{H,l}_t\xi,\md^{W,1}_t\xi,\dots, \md^{W,m}_t\xi\big),\  t\in[0,T]}.$$
We say that $\xi\in \mathbb{D}^{1,p}_{\mathit{loc}}$ if exists a sequence $\{\xi_n(\omega),\Omega_n\}_{n\geq 1}$ such that $\Omega_n\subset \Omega_{n+1}$ for $n\geq 1$, $P\left(\Omega\setminus \left(\bigcup_{n\ge 1}\Omega_n\right)\right)=0$;  $\xi_n\in \mathbb D^{1,p}$ and
$\xi|_{\Omega_n} = \xi_n|_{\Omega_n}$ for all $n\geq 1$.

For the reader convenience we state here the theorem concerning the Mallivian differentiability of the solution to \eqref{eqIntr} in the case of SDE with non-homogeneous coefficients. The proof  is similar to that of \cite[Theorem 2]{ShSh}
\begin{theorem}\label{MalliavinDiff}
Suppose that coefficients $a,b,c$ of \eqref{eqIntr} satisfy the assumptions
\begin{itemize}
\item[B1] for all $t\in[0,T]$ \ $a(t,\cdot),b(t,\cdot)\in C^1(\R^d)$, $c(t,\cdot)\in C^2(\R^d)$;
\item[B2] $a,b,\partial_x a, \partial_x b, \partial_x c, \partial^2_{xx}c$ are bounded;
\item[B3] there exists $\beta>0$ such that for all $t,s\in[0,T]$, $x\in\R^d$
$$
|c(t,x)-c(s,x)|\leq C|t-s|^\beta(1+|x|),\quad |\partial_x c(t,x)-\partial_x c(s,x)|\leq C|t-s|^\beta.
$$
\end{itemize}
Then $X_t\in\bigcap_{p\geq 1} \mathbb D^{1,p}$.
\end{theorem}

\section{Existence of density under simplified H\"ormander condition}
In this section we prove that a solution to \eqref{eqIntr} possesses density of a distribution under a quite strong condition, which we call a simplified H\"ormander condition. More precisely, we will assume in this section that the coefficients of \eqref{eqIntr} satisfy 
\begin{equation}\label{simpHormander}
\mathsf{span}\{c_{\cdot,k}(0,X_0),b_{\cdot,j}(0,X_0)\mid  1\leq k\leq l,1\leq j\leq m\,\}=
\R^d.
\end{equation}

The first step to establish the existence of density is to show the (local) Malliavin differentiability of the solution to \eqref{eqIntr}.
\begin{theorem}
If the coefficients of \eqref{eqIntr} satisfy the assumptions \emph{A1--A4}, then $X_t\in \bigcap_{p\geq 1} \mathbb{D}^{1,p}_{loc} $.
\end{theorem}
\begin{proof}
Define $\Omega_n=\{\omega: \norm{X_\cdot	(\omega)}_{\infty,[0,t]}<n\}$, $n\ge 1$. Obviously,  $\Omega_n\subset \Omega_{n+1}, n\geq 1$ and, since $\norm{X_\cdot	(\omega)}_{\infty,[0,t]}<\infty$ a.s., $\bigcup_{n\geq 1 }\Omega_n=\Omega$. Consider a smooth function $\psi=\psi(x),\,x\in \R$ such that
\begin{itemize}
\item for all $x\in \R$ \ $0\leq \psi(x)\leq 1$;
\item $\psi(x)=1,x\in[-1,1]$;
\item $\psi(x)=0,\, x\notin[-2,2]$.
\end{itemize}
For $n\ge 1$, put  $\Psi_n=\Psi_n(x_1,\ldots,x_d)=(\int_0^{x_1}\psi(y/n)\d y,\ldots,\int_0^{x_d}\psi(y/n)\d y)$, define $d^{(n)}(s,x) = d(t,\Psi_n(x))$, $d\in \set{a,b,c}$, and let $X^{(n)}$  %$\{X^{(n)}_t=(X^{(n),1}_t,\ldots,X^{(n),d}_t),t\geq 0\}$ 
solve
\begin{equation*}
%\begin{aligned}
X^{(n)}_t=X_0  +\int_0^ ta^{(n)}(s,X^{(n)}_s)\d s+\int_0^tb^{(n)}(s,X^{(n)}_s)\d W_s +\int_0^t c^{(n)}(s,X^{(n)}_s)\d B^{H}_s.
%\end{aligned}
\end{equation*}
Since the functions $a_n, b_n, c_n$ satisfy assumptions B1--B3,  in view of Theorem \ref{MalliavinDiff}, $X^{(n)}_t\in \bigcap_{p\geq 1} \mathbb D^{1,p}$. It is not hard to see that that $X^{(n)}_t(\omega)=X_t(\omega)$ for $\omega \in \Omega_n$, which concludes the proof. 
%
%for $n\geq 1$. Indeed,  for all $m\geq 1$ consider the approximative ordinary It\^o  SDEs
%\begin{gather*}
%X^m_t=x_0+\int_0^t a(s,X^m_s)+c(s,X^m_s)\dot Z^m_s\d{s}+\int_0^tb(s,X_s^m)\d{W_s},\\
%X^{(n),m}_t=x_0+\int_0^t a(s,X^{(n),m}_s)+ c(s,X^{(n),m}_s)\dot Z^m_s\d{s}+\int_0^t b(s,X^{(n),m})\d{W_s}.
%\end{gather*}
%Due to the Lemma \ref{Coincide} one has that  $\E{|X_t^{(n),m}-X_t^{m}|^2\indic\{\norm{X^m}_{\infty,[0,t]}<n\}}=0$. Note that $|X_t^{(n),m}-X^m_t|^2\to| X_t^{(n)}-X_t|^2$ and $\norm{X^m}_{\infty,[0,t]}\to\norm{X}_{\infty,[0,t]}$ as $m\to\infty$  in probability, which follows from the uniform convergence in probability $X^{(n),m}$ to $X^{(n)}$ and $X^{m}$ to $X$,(\cite{MiShInvolv}).  Passing to the limit as $m\to \infty$ we get
%$$\E{|X_t-X_t^{(n)}|^2\indic\{\norm{X}_{0,t,\infty}<n\}}\leq \liminf_{m\to \infty} \E{|X^m_t-X^{(n),m}_t|^2\indic\{\norm{X^m}_{0,t,\infty}< n\}}=0.$$
%The last was obtained applying Fatou lemma and lower semicontinuity of $f(y)=\indic\{y< x\}$ for all real $x$. This ends the proof.
\end{proof}
Now we are to prove the main result of this section.
\begin{theorem}
\label{thTOYHorm}
Suppose that the coefficients of \eqref{eqIntr} satisfy assumptions \emph{A1--A4} and the simplified H\"ormander condition \eqref{simpHormander}. Then for all $t>0$ the law of $X_t$ is absolutely continuous with respect to the Lebesgue measure in $\R^d$.
\end{theorem}
\begin{proof}
By the classical condition for existence of density (see e.g. \cite[Theorem 2.1.2]{Nualart}) and thanks to the previous theorem,  it is enough to verify the non-degeneracy of the Mallivain covariation matrix $M(t)=(M_{i,j}(t))_{i,j=1,\ldots,d}$ with $M_{i,j}(t)=\la\md X^i_t,\md X^i_t\ra_{\mathcal H}$. Define the matrix-valued process $J_{t,0}=(J_{t,0}(i,j))_{i,j=1,\ldots,d}$ as the solution to
\begin{equation}
\label{eqJacobian}
\begin{gathered}
J_{t,0}(i,j)=\delta_{i,j}+\sum_{r=1}^d\bigg[\int_0^t \frac{\partial a_i}{\partial x_r}(s,X_s) J_{s,0}(r,j)\d s \\ +\sum_{k=1}^m \int_0^t \frac{\partial b_{i,k}}{\partial x_r}(s,X_s)J_{s,0}(r,j)\d W^{k}_s
+ \sum_{q=1}^l \int_0^t \frac{\partial c_{i,q}}{\partial x_r}(s,X_s)J_{s,0}(r,j) \d B^{H,q}_s\bigg].
\end{gathered}
\end{equation}
where $\delta_{i,j} = \ind{i=j}$ is the Kronecker delta. The system above is linear, hence, possesses a unique solution. In view of  Lemma \ref{lemma1}, $J_{t,0}$ is non-degenerate; denoting $J_{t,s}=J_{t,0}J^{-1}_{s,0}$ and applying Lemma \ref{lemma2} one can write 
\begin{equation*}
%\label{M_t}
\begin{gathered}
M(t)=\sum_{k=1}^m \int_0^t (J_{\newcommand{\Rho}{\mathrm{P}}t,s}b_{\cdot,k}(s,X_s))(J_{t,s}b_{\cdot,k}(s,X_s))'\d s\\
+ 
\sum_{q=1}^l \int_0^t\int_0^t \varphi_H(s,u)(J_{t,s}c_{\cdot,q}(s,X_s))(J_{t,u}c_{\cdot,q}(u,X_u))'\d s\d u=J_{t,0}C_t J'_{t,0},
\end{gathered}
\end{equation*}
where
\begin{equation*}
\begin{gathered}
 C_t=\sum_{k=1}^m \int_0^t (J_{s\newcommand{\Rho}{\mathrm{P}},0}^{-1}b_{\cdot,k}(s,X_s))(J_{s,0}^{-1}b_{\cdot,k}(s,X_s))'\d s \\
+ \sum_{q=1}^l \int_0^t\int_0^t \varphi_H(s,u)(J_{s,0}^{-1}c_{\cdot,q}(s,X_s))(J^{-1}_{u,0}c_{\cdot,q}(u,X_u))'\d s\d u.
\end{gathered}
\end{equation*}
Again, due to the invertibility of $J_{t,0}$,  $M_t$ is invertible if and only if so is $C_t$. Assuming the contrary, there exists a non-zero vector $v\in \R^d$ such that $v'C_tv=0$. Write $$v'C_tv=\sum_{k=1}^m \norm{\la J_{\cdot,0}b_{\cdot,k}(\cdot,X_\cdot),v \ra}^2_{L^2[0,t]}+\sum_{q=1}^l \norm{\la J_{\cdot,0}c_{\cdot,q}(\cdot,X_\cdot),v \ra}^2_{L^2_H[0,t]}.$$ 
Since the functions
\begin{gather*}
s\mapsto \la J^{-1}_{s,0}b_{\cdot,k}(s,X_s,v \ra,\, k=1,\ldots,m,\\
s\mapsto \la J^{-1}_{s,0}c_{\cdot,q}(s,X_s),v\ra,\, q=1,\ldots,l
\end{gather*}
are continuous, they must be equal zero for all $s\in [0,t]$. For  $s=0$ we get
\begin{align*} 
&\sum_{i=1}^d b_{i,k}(0,X_0)v_i=0, k=1,\ldots,m;\\
&\sum_{i=1}^d c_{i,q}(0,X_0)v_i=0, q=1,\ldots,l.
\end{align*}
This, however, contradicts the assumption \eqref{simpHormander}. Consequently, $M_t$ is invertible, as required. 
\end{proof}

\section{Existence of density under strong H\"ormander condition}
In this section we consider a homogeneous version of \eqref{eqIntr}:
\begin{gather}
X_t=X_0+\int_0^t a(X_s)\d s+\int_0^t b(X_s)\d{W_s}+\int_0^t c(X_s)\d{B^H_s}.\label{eqHom}
\end{gather}
In this section we assume that Hurst index $H\in (1/2,2/3)$, and some $\theta\in((H-1/2)/(3-4H),1/2)$ is fixed. The role of the restriction $\theta>(H-1/2)/(3-4H)$ will become clear in the proof of the Norris lemma for \eqref{eqHom} (Lemma~\ref{Chuck}). Now we just remark that the expression $(H-1/2)/(3-4H)$ is increasing for $H\in(1/2,3/4)$ and is equal to $1/2$ for $H=2/3$, so the upper bound $H<2/3$ arises naturally.

We impose the following condition on the coefficients of (\ref{eqHom}):
\begin{itemize}
\item[C1] $a,b,c\in C^\infty_b(\R^d)$ with all derivatives bounded.
\end{itemize}
Under this assumption the solution is infinitely differentiable in the Malliavin sense: $X_t\in \bigcap_{k,p=1}^\infty \mathbb D^{k,p}=\mathbb D^{\infty}$, which can be shown similarly to its differentiability under B1--B3.

The aim of this section is to investigate the existence of a density and properties of this density of a distribution of $X_t$ under the strong H\"ormander condition, which reads as follows.

%In the following we fix some $\theta\in(1-H,1/2)$.

Set $V_0=a$, $V_j(\cdot)=b_{\cdot, j}(\cdot)$ for $j=1,\ldots,m$ and 
$V_{j+m}(\cdot)=c_{\cdot,j}(\cdot)$ for $j=1,\ldots,l$. %Thus, we have a vector field $\Upsilon=\{V_{j}\}_{j=1,\ldots,m+k}$.
Using the Lie bracket $[\cdot,\cdot]$, define the set $$\Upsilon_k=\{[V_{i_1},\ldots,[V_{i_{k-1}},V_{i_k}]\ldots], (i_1,\ldots,i_k)\in\{1,\ldots,d\}^k\}.$$ 
It is said that the vector field $\Upsilon_0=\{V_{j}\}_{j=1,\ldots,m+l}$ satisfies the {H\"ormander condition} at the point $X_0$, if for some positive integer $n_0$ one has
\begin{equation}\label{Hormander}
 \mathsf{span}\set{ V(X_0), V\in \bigcup_{k=1}^{n_0}  \Upsilon_k}=\R^d.
\end{equation}
The main result of this section  is  the following theorem.
\begin{theorem}
 Assume that coefficients of (\ref{eqHom}) satisfy assumption \emph{C1} and the H\"ormander condition \eqref{Hormander}. Then the law of $X_t$ for all $t>0$ possesses  a smooth density with respect to the Lebesgue measure in $\R^d.$
\end{theorem}

\begin{proof}
Using the usual condition for existence of a smooth density (see e.g. \cite[Theorem 2.1.4]{Nualart}) and taking into account that all moments of the Jacobian $J_{t,s}$ are finite, it is enough to show that the matrix inverse to the reduced Malliavin covariance matrix of $X_t$ possesses all moments. 

Recall from Theorem \ref{thTOYHorm} that the reduced Malliavin covariance matrix of the solution to (\ref{eqHom}) can be written as
\begin{gather*}
 C(t)=\sum_{k=1}^m \int_0^t(J_{s,0}^{-1} b_{\cdot,k}(X_s))(J_{s,0}^{-1}b_{\cdot,k}(X_s))'\d s\\
+ \sum_{q=1}^l \int_0^t\int_0^t \varphi_H(s,u)(J_{s,0}^{-1}c_{\cdot,q}(X_s))(J_{u,0}^{-1}c_{\cdot,q}(X_u))'\d s\d u.
\end{gather*}
To simplify the notation, we assume from now that $t=1$.
We are to prove that $\E{|\det C_t|^{-p}}<\infty$ for all $p\geq 1$. Due to \cite[Lemma 2.3.1]{Nualart} it suffices to prove that the entries of $C_t$ possess all moments and for any $p\geq 2$ there exists $C_p$ such that for all $\varepsilon>0$ it holds
 $$\sup_{\norm{v}=1}\pr{\la v, C_1 v\ra\leq \varepsilon}\leq C_p \varepsilon^p.$$
Write
\begin{gather*}
\la v,C_1 v\ra=\sum_{k=1}^m \norm{\la J^{-1}_{\cdot,0}b_{\cdot,k}(X_\cdot),v \ra}^2_{L^2[0,1]}+\sum_{q=1}^l 
\norm{\la J^{-1}_{\cdot,0}c_{\cdot,q}(X_\cdot),v\ra}_{L^2_H[0,1]}.
\end{gather*}
It is well known that $\norm{f}_{L^2_H[0,1]}\le \norm{f}_{L^2[0,1]}$. Therefore, $$\la v,C_1 v\ra\geq C\sum_{k=1}^{m+l}\norm{G_k}_{L^2_H[0,1]}, \text{ where \ }G_k=\la J^{-1}_{\cdot,0} V_{k}(X_\cdot),v \ra.$$
Applying  \cite[Lemma 4.4]{BauHai} we get that 
$$ \la v,C_1 v\ra\ge C\sum_{k=1}^{m+l}\frac{\norm{G_k}^{2(3+1/\theta)}_\infty}{\norm{G_k}^{2(2+1/\theta)}_\theta}$$
for $\theta>H-1/2$. 
Thus, $$\pr{\la v,C_1 v\ra \leq \varepsilon}\leq \pr{C\sum_{k=1}^{m+l}\frac{\norm{G_k}^{2(3+1/\theta)}_\infty}{\norm{G_k}^{2(2+1/\theta)}_\theta}\leq \varepsilon}.$$
From \cite[Lemma 4.5]{BauHai} and Theorem \ref{Chuck} we obtain the following estimate
$$\pr{C\sum_{k=1}^{m+l}\frac{\norm{G_k}^{2(3+1/\theta)}_\infty}{\norm{G_k}^{2(2+1/\theta)}_\theta}\leq \varepsilon}\leq C\varepsilon^p+\min_{k=1,\ldots,m+l}\pr{\norm{\la v, J_{\cdot,0}V_k(X_\cdot)}_\infty\leq \varepsilon^\alpha}.$$
Now let $V$ be a bounded vector field with bounded derivatives of all order. The chain rule implies
\begin{gather*}
 J_{t,0}^{-1}V(X_t)=V(X_0)+\int_0^t J_{s,0}^{-1}([V_0,V]+\frac 1 2 \sum_{k=1}^{m+l}[V_k,[V_k,V]])(X_s)\d s
 \\
 + \sum_{k=1}^m \int_0^t J_{s,0}^{-1}[V_k,V](X_s)\d{W_s}+\sum_{k=m+1}^{l+m}\int_0^t J_{s,0}^{-1}[V_k,V](X_s)\d{B^H_s}.
\end{gather*}
Thus, applying Theorem \ref{Chuck} once more, we obtain
\begin{gather*}
 \pr{\norm{\la v,J_{\cdot,0}V(X_\cdot)\ra}_\infty<\varepsilon}\leq C \varepsilon^p+
 \min_{k=1,\ldots,m+l}\pr{\norm{\la v,J^{-1}_{\cdot,0}[V_k,V](X_\cdot) \ra}_\infty\leq \varepsilon^\alpha}.
\end{gather*}
Let  $n_0$ be the integer from the H\"ormander condition. Iterating our consideration above, we obtain
\begin{gather*}
 \pr{\la v, C_1 v\ra\leq \varepsilon}\leq C\varepsilon^p+
 \min_{V\in\bigcup_{k=1}^{n_0} \Upsilon_k} \pr{\norm{\la v,J^{-1}_{\cdot,0}V(X_\cdot)\ra}_\infty\leq \varepsilon^\alpha}
\end{gather*} 
for all $\varepsilon$ small enough. Since $\{V(x_0), V\in \bigcup_{k=1}^{n_0} \Upsilon_k\}$ spans $\R^d$, there exists $v$ such
that $\la v, V(x_0)\ra\neq 0$. Hence, there exists $\varepsilon_0(p)$ such that for all $\varepsilon<\varepsilon_0(p)$ the second term vanishes. As a result,
$$\pr{\la v, C_1 v\ra\leq \varepsilon}\leq C_p \varepsilon^p$$
 for all $\varepsilon\leq \varepsilon_0(p)$, as required.
\end{proof}

\appendix
\section{Technical lemmas}
The following two lemmas concern the Jacobian of the flow generated by the solution $X$ to equation \eqref{eqIntr}. These are quite standard facts, so we just sketch the proofs.
\begin{lemma}
\label{lemma1}
Under assumptions \emph{A1--A4} the matrix valued process $J_{t,0}=(J_{t,0}(i,j))_{i,j=1,\ldots,d}$ given by \eqref{eqJacobian} has an inverse $Z_{t,0}=(Z_{t,0}(i,j))_{i,j=1,\ldots,d}$ for all $t>0$. Moreover, $\set{Z_{t,0},t\ge 0}$ satisfies the following system of equations
\begin{equation}\label{eqInv}
\begin{gathered}
Z_{t,0}(i,j)=\delta(i,j)-\sum_{r=1}^d\bigg[\int_0^t \frac{\partial a_r}{x_j}(s,X_s) Z_{s,0}(r,j)\d s\\
{} - \sum_{k=1}^m \int_0^t \frac{\partial b_{r,k}}{\partial x_j}(s,X_s)Z_{s,0}(i,r)\d W^{k}_s-
\sum_{q=1}^l \int_0^t  \frac{\partial c_{r,q}}{\partial x_j}(s,X_s)Z_{s,0}(i,r) \d B^{H,q}_s\\ + \frac12\sum_{u=1}^d\sum_{v=1}^d \int_0^t \frac{\partial b_{r,v}}{\partial x_u}(s,X_s)\frac{\partial b_{u,v}}{\partial x_r}(s,X_s)Z_{s,0}(i,r)\d s\bigg].
\end{gathered}
\end{equation}
\end{lemma}
\begin{proof}
The equation \eqref{eqInv} is linear, thus possesses a unique solution $Z_{t,0}$. So we need to verify that $Z_{t,0}J_{t,0}=J_{t,0}Z_{t,0}=I_d$, the identity matrix. 
The equality clearly holds for $t=0$. To show it for $t>0$, it is enough to show that the differentials of $Z_{t,0}J_{t,0}$ and of $J_{t,0}Z_{t,0}$ vanish. 
But this can be routinely checked using the It\^o formula.
%\begin{gather*}
%\sum_{j=1}^d \bigg(-\sum_{\alpha=1}^m\sum_{r=1}^d \int_0^t Z_{s,0}(i,r)J_{s,0}(r,k)\frac{\partial b_{r,\alpha}}{\partial x_j}(s,X_s)\d W^{\alpha}_s+\\
%\sum_{\alpha=1}^m\sum_{r=1}^d \int_0^t J_{s,0}(r,k)Z_{s,0}(i,j)\frac{\partial b_{j,\alpha}}{\partial x_r}(s,X_s)\d W^\alpha_s\bigg)=0
%\end{gather*}
%by rearranging the sums. Thus, $\sum_{j=1}^d Z_{t,0}(i,j)J_{t,0}(j,k)=\delta(i,k)$ for all $i,j,k=1,\ldots,d$.
\end{proof}
Denote for $t\geq s$ $J_{t,s}=J_{t,0}J^{-1}_{s,0}$.
\begin{lemma}
\label{lemma2}
Under assumptions \emph{A1--A4}, the Malliavin derivatives of the solution to \eqref{eqIntr} are given by
\begin{gather}
\md^{W,k}_s X_t=J_{t,s} b_{\cdot,k}(s,X_s)\ind{s\le t},\,k=1,\ldots,m, \label{eqW}\\
\md^{H,q}_s X_t= J_{t,s} c_{\cdot,q}(s,X_s)\ind{s\le t},\, q=1,\ldots,l.
\end{gather}
\end{lemma}
\begin{proof} The argument is exactly the same for both equations, so we  prove only \eqref{eqW}. Evidently, $\md_s^{W,k} X_t = 0$ for $s>t$, so suppose that $s\le t$. Due to the closedness of the derivative, 
we can freely differentiate \eqref{eqIntr} as if the integrals were finite sums, in particular, using the chain rule, we can write for $i=1,\dots,d$
\begin{gather*}
\md_s^{W,k} \int_0^t a_i(u,X_u)\d u =  \int_0^t \md_s^{W,k} a_i(u,X_u)\d u = 
\sum_{r=1}^{d}\int_s^t \frac{\partial}{\partial x_r}a_i(u,X_u)\md_s^{W,k}X^r_u \d u
\end{gather*}
and similarly 
\begin{gather*}
\md_s^{W,k} \int_0^t c_{i,q}(u,X_u)\d B^{H,q}_u =  \sum_{r=1}^{d}\int_s^t \frac{\partial}{\partial x_r}c_{i,q}(u,X_u)\md_s^{W,k}X^r_u \d B^{H,q}_s, q=1,\dots,l,\\
\md_s^{W,k} \int_0^t b_{i,j}(u,X_u)\d W^{j}_u =  \sum_{r=1}^{d}\int_s^t \frac{\partial}{\partial x_r}b_{i,j}(u,X_u)\md_s^{W,k}X^r_u \d W^j_s, j=1,\dots,m, j\neq k.
\end{gather*}
To differentiate the integral w.r.t. $W^k$, approximate it by an integral sum and note that we will have an extra term corresponding to the derivative of the increment of $W^k$ on the interval containing $s$. Passing to the limit, we get
$$\md_s^{W,k} \int_0^t b_{i,k}(u,X_u)\d W^{k}_u =  b_{i,k}(s,X_s) + \sum_{r=1}^{d}\int_s^t \frac{\partial}{\partial x_r}b_{i,k}(u,X_u)\md_s^{W,k}X^r_u \d W^k_s.$$
Therefore, we have for $s\le t$ the following linear equation on $\md_s^{W,k}X_u$:
\begin{gather*}
\md^{W,k}_s X_t = b_{\cdot,k}(s,X_s) + \sum_{r=1}^{d}\bigg[
\int_s^t \frac{\partial}{\partial x_r}a(u,X_u)\md_s^{W,k}X^r_u \d u
\\ + \sum_{q=1}^l \int_s^t \frac{\partial}{\partial x_r}c_{\cdot,q}(u,X_u)\md_s^{W,k}X^r_u \d B^{H,q}_s 
+ \sum_{j=1}^m
\int_s^t \frac{\partial}{\partial x_r}b_{\cdot,j}(u,X_u)\md_s^{W,k}X^r_u \d W^j_s
\bigg].
\end{gather*}
On the other hand, from \eqref{eqJacobian} we can write
\begin{gather*}
J_{t,0}=J_{s,0}+\sum_{r=1}^d\bigg[\int_s^t \frac{\partial a_i}{\partial x_r}(u,X_u) J_{u,0}\d u \\ +\sum_{k=1}^m \int_s^t \frac{\partial b_{\cdot,k}}{\partial x_r}(s,X_s)J_{u,0}\d W^{k}_u
+ \sum_{q=1}^l \int_s^t \frac{\partial c_{\cdot,q}}{\partial x_r}(s,X_s)J_{u,0} \d B^{H,q}_u\bigg],
\end{gather*}
which, upon multiplying by $J_{s,0}^{-1}b_{\cdot,k}(s,X_s)$ on the right leads to the same equation on $J_{u,s} b_{\cdot,k}(s,X_s)$ as that on $\md_s^{W,k}X^r_u$. Hence, by uniqueness, we get the desired result.
\end{proof}
Further we establish a simple estimate on the It\^o integral of a H\"older continuous integrand.
\begin{lemma}
\label{estimInt}
Let $\set{f(t),t\in [0,T]}$ be an $\mathbf{F}$-adapted stochastic process such that $\E{\norm{f}_\theta^p}<\infty $ for all $p\geq 1$, and $0<\delta<\Delta\le T$. Then for all  $s,t,u\in[0,T]$ such that $u<s<t$, $t-s<\delta$, $t-u\le \Delta$ it holds
$$\left|\int_s^t\big(f(v)-f(u)\big)\d{W_v} \right|\leq \Delta^\theta\delta^{1/2}\xi_{\Delta,\delta},$$
where $\E{ \xi^p_{\Delta,\delta}}<C_p\E{\norm{f}_\theta^p}$ for all $p\ge 1$.
\end{lemma}

\begin{proof} It suffices to establish the required result for  $p$ large enough, then one can get deduce it for all $p\ge 1$ with the help of Jensen's inequlaity.

By the Garsia--Rodemich--Rumsey inequality, we get  %for any $\eta\in(0,1/2)$ 
\begin{gather*}
\left|\int_s^t\big(f(v)-f(u)\big)\d{W_v} \right|\leq C|t-s|^{1/4}\left(\int_s^t\int_s^t \frac{|\int_x^y \big(f(v)-f(u)\big)\d{W_v}|^{8}}{|x-y|^{4}} \d x\d y\right)^{1/8}\\
\le C\Delta^\theta \delta^{1/2	}\xi_{\Delta,\delta}, 
\end{gather*}
where
$$
\xi_{\Delta,\delta}=\Delta^{-\theta}\delta^{-1/4}\left(\int_s^t\int_s^t \frac{|\int_x^y \big(f(v)-f(u)\big)\d{W_v}|^{8}}{|x-y|^{4}} \d x\d y\right)^{1/8}.
$$
For $p>8$ the  H\"older inequality entails that
 \begin{gather*}
\E{\xi^p_{\Delta,\delta}}\leq \Delta^{-\theta p}\delta^{-p/4}(t-s)^{2(p/8-1)} \int_s^t\int_s^t \E{\dfrac{|\int_x^y \big(f(v)-f(u)\big)\d{W_v}|^p}{|x-y|^{p/2}}}\d x\d y 
  \\ \leq C_p\Delta^{-\theta p}\delta^{-2} \int_s^t\int_s^t \E{\dfrac{\left(\int_x^y |f(v)-f(u)|^2\d v\right)^{p/2}}{|x-y|^{p/2}}}\d x\d y\\
  \leq C_p \Delta^{-\theta p}\delta^{-2}\E{\norm{f}_\theta^p}\Delta^{p\theta} \delta^2= C_p\E{\norm{f}_\theta^p}.
 \end{gather*}
Hence, we arrive at the desired statement. 
\end{proof}
We also need the result concerning the pathwise regularity property of $X$. It establishes certain exponential integrability of the H\"older seminorm of $X$, so it is an interesting result on its own.
\begin{theorem}
 Let $\{X_t,t\in[0,T]\}$ be the solution to \eqref{eqHom}. Assume that $a,b,c$ satisfy the assumption \emph{C1}. Then $X \in  C^\theta[0,T]$ for $\theta \in (0,1/2)$ and 
 $\E{\exp\set{K\norm{X}_\theta^q}}<\infty$ for all $q\in\big(0,q^*\big)$, $K>0$, where 
$$
q^* = \frac{4H}{2(H+\theta)+1}\wedge \frac{2H+1}{4H}.
$$ 
In particular, $\E{\norm{X}^p_\theta}<\infty$ for all $p> 0$.
\end{theorem}
\begin{proof}
Define for $\eps\in(0,T]$ 
$$
\norm{X}_{\theta,\eps} = \sup_{0\le t-\eps\le s<t\le T}\frac{\abs{X_t-X_s}}{(t-s)^{\theta}}.
$$
Clearly, $\norm{X}_\theta \le \norm{X}_{\theta,\eps} + 2\eps^{-\theta}\norm{X}_\infty$. It follows from \cite[equation (4)]{Integr} that 
$$
\norm{X}_{\theta,\eps} \le C_1\left(\big\|I^b\big\|_\theta + \Lambda_\mu\big(1+ \norm{X}_\infty\eps^{\mu-\theta}\big) \right),
$$
for any $\eps\in(0,C_2 \Lambda_\mu^{-1/\mu}] $ where $C_1,C_2$ are some positive constants, $\mu\in(1/2,H)$, 
$\Lambda_\mu = \norm{B_H}_\mu+1$, $I^b_t = \int_0^t b(X_s)dW_s.$
Therefore, setting $\eps = C_2 \Lambda_\mu^{-1/\mu}$, we obtain
$$
\norm{X}_{\theta} \le C_1\left(\big\|I^b\big\|_\theta + \Lambda_\mu+ 2 \norm{X}_\infty\Lambda^{\theta/\mu} \right) \le 
C\left(\big\|I^b\big\|_\theta + \Lambda_\mu+  \norm{X}_\infty^{p'} + \Lambda^{q'\theta/\mu} \right),
$$
where $p'>1$, and $q' = p'/(p'-1)$ is the exponent conjugate to $p'$.
Therefore,
$$
\norm{X}_{\theta}^q \le  
C\left(\big\|I^b\big\|_\theta^q + \Lambda_\mu^q+  \norm{X}_\infty^{p'q}+ \Lambda^{qq'\theta/\mu} \right).
$$
Evidently,  $q^*<1$, so it follows from \cite[Lemma 1]{ShSh} that $\Ex{\exp\set{K\big\|I^b\big\|_\theta^q}}<\infty$ for all $K>0$. Further, $\Lambda_\mu$ is an almost surely finite supremum of a centered Gaussian family, so by Fernique's theorem 
$\Ex{\exp\set{K \Lambda_\mu^z}}<\infty$ for any $K>0$, $z\in(0,2)$. Finally, by \cite[Corollary 4]{Integr}, $\Ex{\exp\set{K \norm{X}^{z}_\infty}}<\infty$ for all $K>0, z<4H/(2H+1)$. Now if $p'>1$ is close to $4H q^{-1}(2H+1)^{-1}$ (thanks to the bound on $q$ such choice is possible) and $\mu$ is close to $H$, then $q'$ is close to $4H/(4H-q(2H+1))$, and $qq'\theta/\mu$ is close to 
$4q\theta/(4H-q(2H+1))$, which is less than $2$. Indeed, the last statement is equivalent to $q(2\theta+2H+1)<4H$, which is true thanks to the restriction on $q$. Thus, we get the desired integrability.
\end{proof}

%\begin{lemma}
%Suppose that functions $a,b,c$ and $\tilde{a},\tilde{b},\tilde{c}$ satisfy assumptions \textbf{A.1,A.2,A.3}. Moreover $d(s,x)=\tilde{d}(s,x)$ for all $s\geq 0$, $\abs{x}\leq N$ and $d\in\{a,b,c\}$. Then the solutions of the following SDEs 
%\begin{gather*}
%X_t=x_0+\int_0^ta(s,X_s)+\dot{Z}^m_sc(s,X_s)\d{s}+\int_0^tb(s,X_s)\d{W_s},\\
%Y_t=x_0+\int_0^t\tilde{a}(s,Y_s)+\dot{Z}^m_s \tilde{c}(s,Y_s)\d{s}+\int_0^t\tilde{b}(s,Y_s)\d{W_s}.
%\end{gather*}
%for all $m\geq 1$ coincide on the set $\Omega_N(t)=\{\omega:\norm{X(\omega)}_{0,t,\infty}< N\}$, that is $X_t(\omega)=Y_t(\omega), \omega\in \Omega_N(t)$ a.s. Here $Z^{m}_t=m\int_{0\vee (t-1/m)}^t B^H_s\d{s}$.
%\label{Coincide}
%\end{lemma}
%\begin{proof}
%One can write 
%\begin{gather*}
%\E{|X_t-Y_t|^2\ind{\Omega_N(t)}}\leq C\Bigl(\int_0^t\E{(a(s,X_s)-\tilde{a}(s,Y_s))^2\ind{\Omega_N(s)}}\d{s}\\ + 
%\int_0^t\E{(\dot Z^m_s)^2}\E{(c(s,X_s)-\tilde{c}(s,Y_s))^2\ind{ \Omega_N(s)}}\d{s}+\\
%\E{(\int_0^t b(s,X_s)-\tilde{b}(s,Y_s)\d{W_s})^2\ind{\Omega_N(t)}}\Bigr)\leq C\int_0^t\E{|X_s-Y_s|^2\ind{\Omega_N(s)}}\d{s}.
%\end{gather*}
%last was obtained using assumptions on the coefficients, boundedness of moments of $\dot Z^m_s$, It\^o isometry]. Applying Gronwall lemma we conclude the statement of lemma.
%\end{proof}

The following result is a Norris type lemma for mixed SDEs. It is a crucial result to prove existence of density under the H\"ormander condition. Loosely speaking, this statement says that if 
\begin{equation}
Y_t=Y_0+\int_0^t a(s)\d{s}+\int_0^t b(s)\d{W_s}+\int_0^tc(s)\d{B^H_s},\label{NorrisEq}
\end{equation}
 $\norm{Y}_\infty  = \norm{Y}_{\infty; [0,T]}$ is small, then $\norm{b}_\infty$ and $\norm{c}_\infty$ can not be large. This means that the integral w.r.t. $W$ and w.r.t. $B^H$ can not compensate each other well. The rigorous formulation is as follows.
\begin{theorem}
\label{Chuck}
Assume that $H\in(1/2,2/3)$, $\theta\in\big(\theta_*,1/2\big)$,
where 
$$
\theta_* =  \frac{H-\frac 12}{3-4H},
$$
 and that $a,b,c$ in \eqref{NorrisEq}  are $\mathbf F$-adapted processes  satisfying $\E{\norm{a}^p_\infty+\norm{b}^p_\theta+\norm{c}^p_\theta}<\infty$ for all $p\geq 1$. Then exists $q>0$ such that for all $p\geq 1, \varepsilon>0$  
$$\pr{\norm{Y}_\infty<\varepsilon \textrm{ and } \norm{b}_\infty+\norm{c}_\infty>\varepsilon^q}\leq C_p\varepsilon^p.$$
\end{theorem}   
\begin{proof}
Here we imitate the proof of in \cite[Proposition 3.4]{BauHai}.
For notational simplicity, we assume that $T=1$. For some positive integers $M$ and $r$ denote $\Delta = 1/M$, $\delta = \Delta/r$ and
define the following uniform partitions of $[0,1]$: $T_N = N\delta$, $N=0,\dots,M$; $t_n=\delta n$, $n=0,\ldots, Mr$. Further, fix some $\breve H\in (1/2,H)$ and write for $N=0,\dots,M-1$, $n= Nr,\dots, (N+1)r-1$ (so that $t_n\in [T_N,T_{N+1})$),  $i=1,\ldots, d$
\begin{equation}\label{eqS}
\begin{gathered}
\big\la c_i(T_N), B^H_{t_{n+1}}-B^H_{t_n} \big\ra+
\big\la b_{i}(T_N), W_{t_{n+1}}-W_{t_n}\big\ra 
\leq |Y^i_{t_{n+1}}-Y^i_{t_n}|+ \delta\norm{a}_\infty \\
+ \abs{\int_{t_n}^{t_{n+1}}\la b_i(s)-b_i(T_N),\d{W_s}\ra}+\abs{\int_{t_n}^{t_{n+1}}\la c_i(s)-c_i(T_N), \d{B^H_s}\ra }\\ \leq 2\norm{Y}_\infty+\delta\norm{a}_\infty +C\Delta^\theta\delta^{\breve{H}}\norm{c}_\theta\norm{B^H}_{\breve{H}}+\Delta^{\theta}\delta^{1/2}\xi_{\Delta,\delta}=:S,
\end{gathered}
\end{equation}where in the last step we have used the Young--Love inequality (see e.g. \cite[Proposition 1]{NualSau}) and {Lemma \ref{estimInt}}.

 %Now denote the the right hand side of the last inequality by $\mathit S$. 
 For processes $\xi,\zeta$ denote
%\begin{gather*}
%V_N^{u,v}(B^H)=\sum_{n=(N-1)r}^{Nr-1} (B^{H,u}_{t_{n+1}}-B^{H,u}_{t_n})(B^{H,v}_{t_{n+1}}-B^{H,v}_{t_n}),\text{ for } u,v=1\ldots,l.\\
%V_N^{u,v}(W)=\sum_{n=(N-1)r}^{Nr-1} (W^u_{t_{n+1}}-W^u_{t_n})(W^v_{t_{n+1}}-W^v_{t_n}),\text{ for } u,v=1\ldots,m.\\
%V^{u,v}_N()=\sum_{n=(N-1)r}^{Nr-1} (W^u_{t_{n+1}}-W^u_{t_n})(B^{H,v}_{t_{n+1}}-B^{H,v}_{t_n}),\text{ for } u=1\ldots,m, v=1,\ldots, l.
%\end{gather*}
\begin{equation}
\label{V_N}
V_N(\xi,\zeta) = \sum_{n=Nr}^{(N+1)r-1} \left(\xi_{t_{n+1}}-\xi_{t_n}\right)\left(\zeta_{t_{n+1}}-\zeta_{t_n}\right);
\end{equation}
we remind that the summation is in fact over $t_n\in[T_N,T_{N+1})$.
Squaring the both sides of \eqref{eqS}, summing  over $n= Nr,\dots, (N+1)r-1$ and then taking the square root we get
\begin{gather}
 \biggl(\sum_{u,v=1}^m \sum_{v=1}^m b_{i,u}(T_N)b_{i,v}(T_N)V_N(W^u,W^v) + \sum_{u,v=1}^l  c_{i,u}(T_N)c_{i,v}(T_N)V_N(B^{H,u},B^{H,v})
\\
+ \sum_{u=1}^m \sum_{v=1}^l b_{i,u}(T_N)c_{i,v}(T_N)V_N(W^u,B^{H,v}) \biggr)^{1/2}\leq C \Delta^{1/2}\delta^{-1/2}\mathit{S}.
\end{gather}
Therefore,
\begin{equation}\label{sumincrements}
\begin{gathered}
\sum_{u=1}^m \abs{b_{i,u}(T_N)}V_N(W^u,W^u)^{1/2}+\sum_{v=1}^l \abs{c_{i,v}(T_N)}V_N(B^{H,v},B^{H,v}) \\
\le  C\biggl ( \sum_{1\le u<v\le m}|b_{i,u}(T_N)|^{1/2}|b_{i,v}(T_N)|^{1/2}|V_N(W^u,W^v)|^{1/2}\\ +\sum_{1\le u<v\le l}|c_{i,u}(T_N)|^{1/2}|c_{i,v}(T_N)|^{1/2}|V_N(B^{H,u},B^{H,v})|^{1/2}\\
+\sum_{u=1}^m \sum_{v=1}^l |b_{i,u}(T_N)|^{1/2}|c_{i,v}(T_N)|^{1/2}|V_N(W^u,B^{H,v}|^{1/2}+\Delta^{1/2}\delta^{-1/2}\mathit S \biggr).
\end{gathered}
\end{equation}
Further, for arbitrary $f\in C^\theta[0,1]$,
$$\abs{\Delta\sum_{N=0}^{M-1}|f(T_N)|-\norm{f}_{L^1[0,1]}}\leq \norm{f}_\theta\Delta^{\theta},$$
which yields 
\begin{gather*}
\sum_{u=1}^m\norm{b_{i,u}}_{L^1[0,1]}
+\sum_{v=1}^l\norm{c_{i,v}}_{L^1[0,1]}
\\
\leq \sum_{u=1}^m\bigg(\Delta^{\theta}\norm{b_{i,u}}_\theta+ 
\Delta\sum_{N=0}^{M-1} \abs{b_{i,u}(T_N)}\bigg) + \sum_{v=1}^l\bigg(\Delta^{\theta}\norm{c_{i,v}}_\theta+ 
\Delta\sum_{N=0}^{M-1} \abs{c_{i,v}(T_N)}\bigg)\\
\leq \sum_{u=1}^m\bigg(\Delta^{\theta}\norm{b}_\theta+ 
\Delta^{1/2}\norm{b}_\infty\sum_{N=0}^{M-1} \abs{\Delta^{1/2} - V_N(W^u,W^u)^{1/2}}
\\+ 
\Delta^{1/2}\delta^{1/2-H}\sum_{N=0}^{M-1} \abs{b_{i,u}(T_N)} V_N(W^u,W^u)^{1/2}\bigg)
\\
+ \sum_{v=1}^l\bigg(\Delta^{\theta}\norm{c}_\theta+ 
\Delta^{1/2}\delta^{1/2-H}\norm{c}_\infty\sum_{N=0}^{M-1} \abs{\Delta^{1/2}\delta^{H-1/2} - V_N(B^{H,v},B^{H,v})^{1/2}}\\
  + 
 \Delta^{1/2}\delta^{1/2-H}\sum_{N=0}^{M-1}\abs{c_{i,v}(T_N)}
 V_N(B^{H,v},B^{H,v})^{1/2}\bigg).
\end{gather*}
Therefore, using \eqref{sumincrements}, we arrive at
\begin{gather*}
\norm{b}_{L^1[0,1]}+\norm{c}_{L^1[0,1]}\leq C\bigg ( \Delta^{\theta}\big(\norm{b}_\theta+ \norm{c}_\theta\big)\\
+\Delta^{1/2}\delta^{1/2-H} \norm{b}_\infty\sum_{N=0}^{M-1}\sum_{u,v=1}^m \abs{\Delta^{1/2}\delta_{u,v} - \abs{V_N(W^u,W^v)}^{1/2}}\\ +\Delta^{1/2}\delta^{1/2-H} \norm{c}_\infty\sum_{N=0}^{M-1}\sum_{u,v=1}^l \abs{\Delta^{1/2}\delta^{H-1/2}\delta_{u,v} - \abs{V_N(B^{H,u},B^{H,v})}^{1/2}}\\
+ \Delta^{1/2}\delta^{1/2-H}\norm{c}^{1/2}_\infty
\norm{b}^{1/2}_\infty
\sum_{N=0}^{M-1}\sum_{u=1}^m\sum_{v=1}^l  \abs{V_N(W^u,B^{H,v})}^{1/2}
+\delta^{-H}S\bigg)\\
\le C\bigg ( \norm{b}_\theta+ \norm{c}_\theta+
\Delta^{-1/4}\delta^{3/4-H} \norm{b}_\infty R^W\\ +\Delta^{H-1}\delta^{1-H} \norm{c}_\infty R^B
+ \Delta^{-1/4}\delta^{(1-H)/2}\left(\norm{c}_\infty +
\norm{b}_\infty\right)R^{W,B}
+\delta^{-H}S\bigg),
\end{gather*}
where
\begin{equation}\label{RWB}
\begin{gathered}
R^W = \Delta^{3/4}\delta^{-1/4} \sum_{N=0}^{M-1}\sum_{u,v=1}^m \abs{\Delta^{1/2}\delta_{u,v} - \abs{V_N(W^u,W^v)}^{1/2}},\\
R^B = \Delta^{H-3/2}\delta^{1/2}\sum_{N=0}^{M-1}\sum_{u,v=1}^l  \abs{\Delta^{1/2}\delta^{H-1/2}\delta_{u,v} - \abs{V_N(B^{H,u},B^{H,v})}^{1/2}},\\
R^{W,B} = \Delta^{3/4}\delta^{-H/2} \sum_{N=0}^{M-1}\sum_{u=1}^m
\sum_{v=1}^{l} \abs{V_N(W^u,B^{H,v})}^{1/2}.
\end{gathered}
\end{equation}
Further we use the following interpolation inequality, valid for any $f\in C^\theta[0,1]$ and $\gamma<1$:
$$\norm{f}_\infty\leq C\big(\gamma \norm{f}_\theta+\gamma^{-1/\theta}\norm{f}_{L^1[0,1]}\big).$$
for any $\gamma\leq1$. Thus,
\begin{equation}\label{bcinfty}
\begin{gathered}
\norm{b}_\infty+\norm{c}_\infty\leq C\big(\norm{b}_\theta + \norm{c}_\theta\big)\gamma + 
C\gamma^{-1/\theta}\Big(\big(\norm{b}_\theta + \norm{c}_\theta\big)\Delta^\theta  \\
+ (\norm{b}_\infty + \norm{c}_\infty)\Big[ 
\Delta^{-1/4}\delta^{3/4-H}R^W + \Delta^{H-1}\delta^{1-H}R^B + \Delta^{-1/4}\delta^{(1-H)/2} R^{W,B}\Big] \\
+ \delta^{-H}\norm{Y}_\infty + \delta^{1-H}  \norm{a}_\infty +  \Delta^\theta \delta^{\breve H-H}\norm{c}_\theta\norm{B^H}_{\breve H} + \Delta^\theta \delta^{1/2-H} \xi_{\Delta,\delta}\Big).
\end{gathered}
\end{equation}
Now we want to put
\begin{gather*}
\Delta^{\beta}\sim \varepsilon^\beta, \delta\sim \varepsilon^{\alpha}, \gamma\sim \varepsilon^\eta, \alpha>\beta>0, \eta>0,
\end{gather*}
so that in the right-hand side of \eqref{bcinfty}, the exponents of $\varepsilon$ are positive for all terms except $\norm{Y}_\infty$. Since $(H-1/2)/\theta\le (3-4H)<1$, it is possible to take $\beta/\alpha\in \big((H-1/2)/\theta, (3-4H)\big)$ so that both  $\theta \beta + (1/2-H)\alpha$ and  $-\beta/4 + (3/4-H)\alpha$ are positive. Also $(H-1)\beta +(1-H)\alpha = (1-H)(\alpha-\beta)>0$, 
$-\beta/4 + (1-H)\alpha/2 > -\beta/4 + (3/4-H)\alpha>0$, $\theta \beta + (\breve H - H)\alpha > \theta \beta + (1/2-H)\alpha>0$. Therefore, by choosing $\eta$ small enough we can make all needed exponents positive.

Thus, for some $\kappa>0$ and $C_1>0$ we have
\begin{gather*}
\norm{b}_\infty+\norm{c}_\infty\leq C_1\norm{Y}_\infty \eps^{-\lambda} + 
C_1\eps^\kappa\Big(  (\norm{b}_\infty + \norm{c}_\infty)\big[ 
R^W + R^B + R^{W,B}]\\ + \norm{b}_\theta + \norm{c}_\theta  
+   \norm{a}_\infty +  \norm{c}_\theta\norm{B^H}_{\breve H} +  \xi_{\Delta,\delta}\Big),
\end{gather*}
where $\lambda= H\alpha+\eta/\theta$. Consequently, for $\eps$ small enough
\begin{gather*}
\pr{\norm{b}_\infty+\norm{c}_\infty>\varepsilon^{\kappa/2} \text{ and } \norm{Y}_\infty<\varepsilon^{\lambda+\kappa}} \\
\le \pr{R^W \ge \eps^{-\kappa/3}} + \pr{R^B \ge \eps^{-\kappa/3}} +  \pr{R^{W,B} \ge \eps^{-\kappa/3}} \\+ \pr{\norm{b}_\theta + \norm{c}_\theta + \norm{c}_\theta\norm{B^H}_{\breve H} +  \xi_{\Delta,\delta}\ge \eps^{-\kappa/3}}.
\end{gather*}
Now the statement follows by applying Lemmas \ref{estimInt} and \ref{concentrat} and the Chebyshev inequality. 
\end{proof}

\begin{lemma}\label{concentrat}
Let $R^W,R^B$ and $R^{W,B}$ be given by \eqref{RWB} and \eqref{V_N}. Then we have for any $h>1$ the following concentration inequalities
\begin{align}
\pr{R^{W}\geq h}&\leq \dfrac{C}{\Delta}\exp(-Ch^2),\label{conc1}\\
\pr{R^B\geq h}&\leq \dfrac{C}{\Delta}\exp(-Ch^2),\label{conc2}\\
\pr{R^{W,B}\geq h}&\leq \dfrac{C}{\Delta}\exp(-Ch^2)\label{conc3}.
\end{align}
\end{lemma}
\begin{proof}

By \cite[Lemma 3.1]{BauHai} we have for $h>0$
\begin{equation}\label{WvWv}
\pr{\big|\Delta^{1/2}-V_N(W^u,W^u)^{1/2}\big|\Delta^{-1/4}\delta^{-1/4}\geq h}\leq C\exp\left(-C {h^2}\right).
\end{equation}
Further, let $u\neq v$. Since  $W^u$ and $W^v$ are independent, and $W^v$ has independent increments, then conditional on $W^v$,  $V_N(W^w,W^u)\Delta^{-1/2}\delta^{-1/2}$ has a centered Gaussian distribution with the variance $ V_N(W^v,W^v)\Delta^{-1} $. Therefore,
\begin{gather*}
\pr{|V_N(W^u,W^v)|^{1/2}\Delta^{-1/4}\delta^{-1/4}\geq h}\\ = \E{\pr{|V_N(W^u,W^v)|\Delta^{-1/2}\delta^{-1/2}\geq h^2}\big|\ W^v}
\le C\,\E{\exp\set{-\frac{h^4\Delta}{4 V_N(W^v,W^v)}}}\\
\le C\exp\set{-\frac{h^4\Delta}{4(\Delta^{1/2}+v)^2}} + \pr{\big|\Delta^{1/2}-V_N(W^v,W^v)^{1/2}\big|\ge v}\\
\le C\exp\set{-\frac{h^4\Delta }{8(\Delta+v^2)}} + C\exp\set{-C\frac{v^2}{\Delta^{1/2}\delta^{1/2}}}, 
\end{gather*}
where we have used \eqref{WvWv}.
Setting $v^2 = h^2\Delta$ and recalling that $\Delta\ge\delta$ we get
$$ \pr{|V_N(W^u,W^v)|^{1/2}\Delta^{-1/4}\delta^{-1/4}\geq h}\leq C\exp\left(-C {h^2}\right).
$$
Combining this with \eqref{WvWv}, we get
\begin{gather*}
\pr{R^W\ge h}\le \sum_{N=0}^{M-1} \sum_{u,v=1}^m \pr{\Delta^{-1/4}\delta^{-1/4} \big|\Delta^{1/2}\delta_{u,v}-\abs{V_N(W^u,W^v)}^{1/2}\big|\ge hm^2}\\
\le \frac{C}{\Delta} \exp\set{-Ch^2}.
\end{gather*}
Using the inequalities from \cite[Lemma 3.2]{BauHai} and repeating the last step, we get \eqref{conc2}.

The estimate \eqref{conc3} is proved similarly to \eqref{conc1}, so we omit some details. Write
\begin{gather*}
\pr{|V_N(W^u,B^{H,v})|^{1/2}\Delta^{-1/4}\delta^{-H/2}\geq h}\\ = \E{\pr{|V_N(W^u,B^{H,v})|\Delta^{-1/2}\delta^{-H}\geq h^2}\big|\ B^{H,v}}
\le C\,\E{\exp\set{-\frac{h^4\Delta\delta^{2H-1}}{4 V_N(B^{H,v},B^{H,v})}}}\\
\le C\exp\set{-\frac{h^4\Delta\delta^{2H-1}}{4(\Delta^{1/2}\delta^{H-1/2}+v)^2}} + \pr{\big|\Delta^{1/2}\delta^{H-1/2}-V_N(B^{H,v},B^{H,v})^{1/2}\big|\ge v}\\
\le C\exp\set{-\frac{h^4\Delta\delta^{2H-1} }{8(\Delta\delta^{2H-1}+v^2)}} + C\exp\set{-C\frac{v^2}{4\Delta^{2H-1}\delta}}.
\end{gather*}
Setting $v^2 = h^2 \Delta \delta^{2H-1}$ and taking into account that $\Delta\ge \delta$, we arrive at
$$
\pr{|V_N(W^u,B^{H,v})|^{1/2}\Delta^{-1/4}\delta^{-H/2}\geq h}\le C\exp\set{-Ch^2}.
$$
From here \eqref{conc3} is deduced similarly to \eqref{conc1}.
\end{proof}

\end{document}